\documentclass[12pt]{article}
\usepackage{amsmath,amssymb,amsbsy,amsfonts,amsthm,latexsym,
amsopn,amstext,amsxtra,euscript,amscd,graphicx}
\newtheorem{prop}{Property}

\newtheorem{lem}{Lemma}
\theoremstyle{definition}

\newtheorem*{rem}{Remark}

\begin{document}
\title{The $f$-belos}

\author{
{\sc Antonio M. Oller-Marc\'{e}n}\\
{Centro Universitario de la Defensa}\\
{Academia General Militar}\\ {Ctra. de Huesca, s/n, 50090 Zaragoza, Spain}\\
{oller@unizar.es}}

\date{}

\maketitle

\begin{abstract}
The \emph{arbelos} is the shape bounded by three mutually tangent semicircles with collinear diameters. Recently, Sondow introduced the parabolic analog, the \emph{parbelos} and proved several properties of the parbelos similar to properties of the arbelos. In this paper we give one step further and generalize the situation considering the figure bounded by (quite) arbitrary similar curves, the \emph{$f$-belos}. We prove analog properties to those of the arbelos and parbelos and, moreover, we characterize the parbelos and the arbelos as the $f$-beloses satisfying certain conditions.
\end{abstract}

\section{Introduction}

The \emph{arbelos} ($\acute{\alpha}\rho\beta\eta\lambda o\varsigma$, literally ``shoemaker's knife'') was introduced in Proposition 4 of Archimedes' Book of Lemmas \cite[p. 304]{ARC}. It is the plane figure bounded by three pairwise tangent semicircles with diameters lying on the same line (see the left-hand side of Figure \ref{arbparb}). In addition to the properties proved by Archimedes himself, there is a long list of properties satisfied by this figure. Boas's paper \cite{BOA} presents some of them and is a good source for references.

It is quite surprising to discover that for 23 centuries no generalizations of this figure were introduced. Recently, Sondow \cite{SON} has extended the original construction considering latus rectum arcs of parabolas instead of semicircles (see right-hand side of Figure \ref{arbparb}). In his paper, Sondows proves several interesting properties of his construction (named \emph{parbelos}) that are, in some sense, counterparts of properties of the arbelos. 

\begin{figure}[h]
\label{arbparb}
\begin{center}
\includegraphics[width=\linewidth]{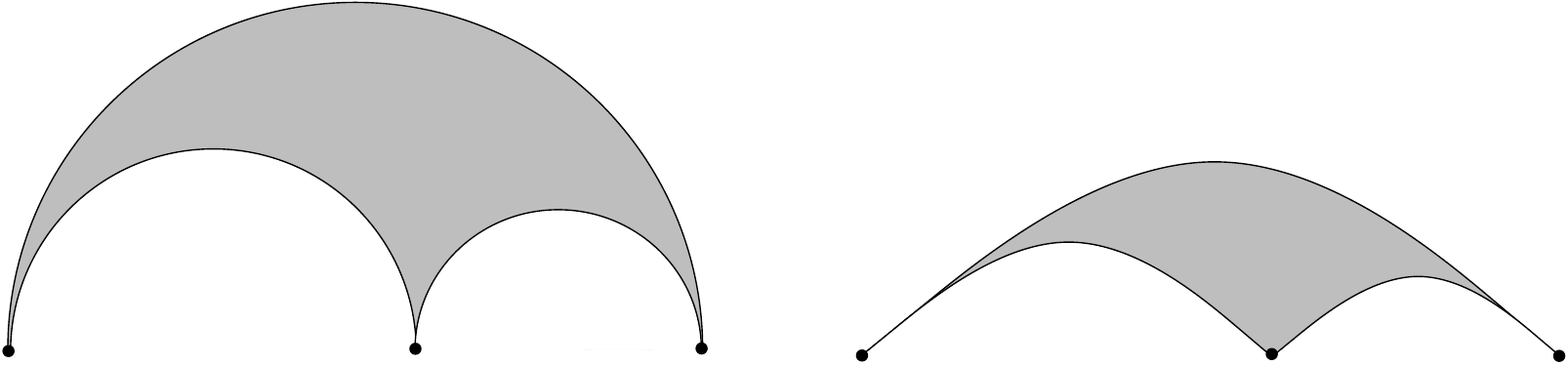}
\end{center}
\caption{An arbelos (left) and a parbelos (right).}
\end{figure}

The motivation for considering latus rectum arcs of parabolas instead of semicircles is clear. Just like all circles are similar, so to all parabolas are similar. Of course this is a very special property of these curves which is not shared even by other conics. Nevertheless, it gives the clue for a further generalization of both the arbelos and the parbelos. In a suitable generalization, the three considered arcs must be similar curves.

In Section 2 we present our construction and the resulting family of figures that we shall call \emph{$f$-belos}. Subsequent sections are mainly devoted to extend and give analogs of some the properties found in \cite{SON}. In passing we will see how arbelos and parbelos appear as particular cases of our construction imposing certain seemingly unrelated conditions.

\section{The $f$-belos}

Let $f:[0,1]\longrightarrow \mathbb{R}$ be a function such that $f(x)>0$ except for $f(0)=f(1)=0$. We will assume that $f$ is continuous in $[0,1]$ and differentiable in $(0,1)$. Given $p\in (0,1)$ we define functions $g:[0,p]\longrightarrow\mathbb{R}$ and $h:[p,1]\longrightarrow\mathbb{R}$ given by:
$$g(x)=pf(x/p),$$
$$h(x)=(1-p)f\left(\frac{x-p}{1-p}\right).$$
Observe that both $g$ and $h$ are similar to $f$ ($g$ is obtained by a homothety centered at the origin and $h$ is obtained by a homothety followed by a translation). In what follows we will consider the case when the graphs of $g$ and $h$ are below the graph of $f$ so that a situation like the one in Figure \ref{efe} makes sense.

\begin{figure}[h]
\label{efe}
\begin{center}
\includegraphics{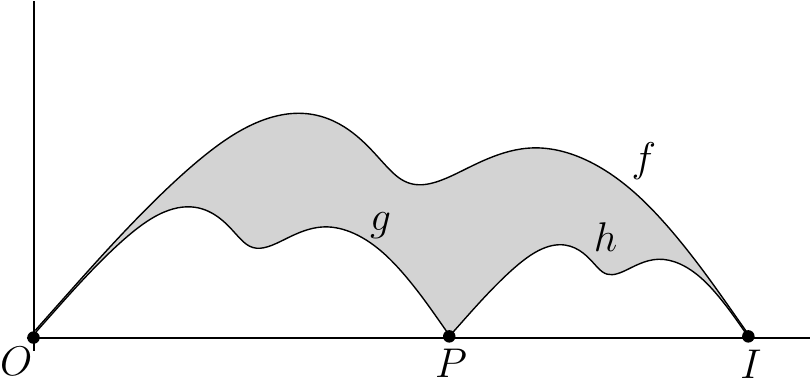}
\end{center}
\caption{The $f$-belos.}
\end{figure}

Given the function $f$ we will call the figure obtained in the previous construction an $f$-belos. The point $P=(p,0)$ will be called the cusp of the $f$-belos. We will denote $O=(0,0)$ and $I=(1,0)$ (see Figure \ref{efe}).

Observe that if $f(x)=\sqrt{x-x^2}$ we recover the original arbelos, while if $f(x)=x-x^2$ we obtain Sondow's parbelos.

\section{Elementary properties of the $f$-belos}

In spite of the generality of the latter construction, the $f$-belos satisfy several interesting properties which, in some sense, extend those of the arbelos and the parbelos. These properties are analogs of Properties 1 and 2 in \cite{SON}.

\begin{prop}
The upper and lower boundaries of an $f$-belos have the same length.
\end{prop}
\begin{proof}
If $L_f$ is the length of the upper boundary, $L_g$ is the length of the  lower arc corresponding to the graph of $g$ and $L_h$ in the length of the lower arc corresponding to the graph of $h$ it follows by their similarity that $L_g=pL_f$ and $L_h=(1-p)L_f$. Hence, the result.
\end{proof}

The following lemma is easy to prove and it is closely related to Plato's analogy of the line \cite{BAL}.

\begin{lem}
Consider a segment $\overline{AB}$ and choose any point $C\in\overline{AB}$ except the endpoints. Now, let $D\in\overline{AC}$ and $E\in\overline{CB}$ be points such that
$$\frac{|AC|}{|CB|}=\frac{|AD|}{|DC|}=\frac{|CE|}{|EB|}.$$
Then, 
$$|DC|=|CE|=\dfrac{|AC|\cdot |CB|}{|AB|}.$$
\end{lem}

As a consequence we obtain the following property.

\begin{prop}
Under each lower arc of an $f$-belos, construct a new $f$-belos similar to the original. Of the four lower arcs, the middle two are congruent, and their common length equals one half the harmonic mean of the lengths of the original lower arcs.
\end{prop}
\begin{proof}
It is enough to apply the previous lemma noting that the lengths of the considered arcs are proportional to the length of the horizontal segment that it determines.
\end{proof}

\section{The parallelogram associated to a point}

Let $x_0\in (0,1)$ and consider the point $P_1=(x_0,f(x_0))$. This point lies on the graph of $f$ and hence, by similarity, it has corresponding points $P_2$ and $P_3$ in the graphs of $g$ and $h$, respectively. Namely
$P_2=(px_0,pf(x_0))$ and $P_3=((1-p)x_0+p,(1-p)f(x_0))$ (see Figure \ref{par}). Observe that $\vec{P_2P_1}=\vec{PP_3}=(1-p)(x_0,f(x_0))$. Hence, $P_1P_2PP_3$ is a parallelogram. Since it depends on the choice of $x_0$, we will denote this parallelogram $\mathcal{P}(x_0)$.

\begin{figure}[h]
\label{par}
\begin{center}
\includegraphics{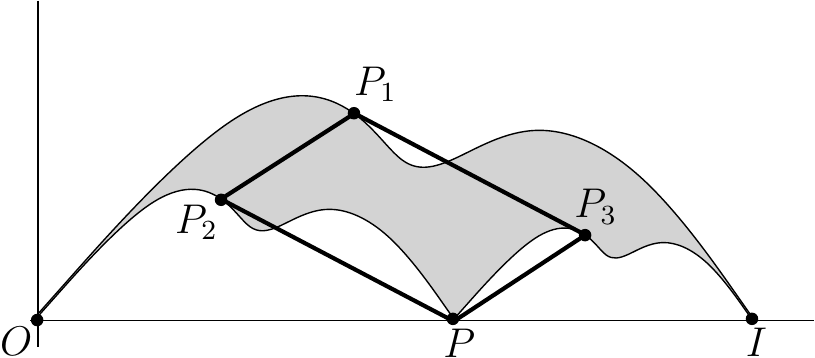}
\end{center}
\caption{The parallelogram associated to the point $P_1$.}
\end{figure}

It is interesting to study when $\mathcal{P}(x_0)$ is a rectangle. This leads to a surprising characterization of the arbelos.

\begin{prop}
Given an $f$-belos, the parallelogram $\mathcal{P}(x_0)$ is a rectangle if and only if $f(x_0)^2=x_0-x_0^2$. Consequently, $\mathcal{P}(x_0)$ is a rectangle for every $x_0\in (0,1)$ if and only if $f$ describes a semicircle (i.e., the figure is an arbelos).
\end{prop}
\begin{proof}
We have that $\vec{P_2P}=(p-px_0,-pf(x_0))$. Then, $\mathcal{P}(x_0)$ is a rectangle if and only if $\vec{P_2P_1}\bot \vec{P_2P}$; i.e., if and only if
$$0=\vec{P_2P}\cdot\vec{P_2P_1}=p(1-p)[x_0(1-x_0)-f(x_0)^2].$$
\end{proof}

The area of the parallelogram $\mathcal{P}(x_0)$ can be easy computed by:
$${\rm area}(\mathcal{P}(x_0))=||\vec{P_2P_1}\times\vec{P_2P}||=||(0,0,-p(1-p)f(x_0))||=p(1-p)f(x_0).$$
This fact leads to the following property.

\begin{prop}
Given an $f$-belos, let $c\in(0,1)$ be such that $f(c)$ is the mean value of $f$ on $[0,1]$ and let $\mathcal{P}(c)$ be the parallelogram associated to $c$ (see Figure \ref{mean}) . Then:
$${\rm area}(f\textrm{-belos})=2{\rm area}(\mathcal{P}(c)).$$
\end{prop}
\begin{figure}[h]
\label{mean}
\begin{center}
\includegraphics{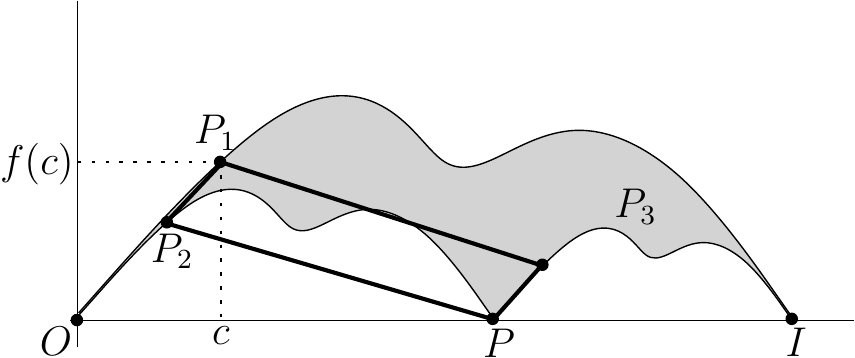}
\end{center}
\caption{The parallelogram associated to the point $(c,f(c))$.}
\end{figure}
\begin{proof}
Let us denote by $A$ the area below the upper arc of the $f$-belos; i.e.:
$$A=\int_0^1 f(x)\ dx.$$
By similarity, the area below the lower arc corresponding to the graph of $g$ is $p^2A$, while the area below the lower arc corresponding to the arc of $h$ is $(1-p)^2A$. Hence the area of the $f$-belos is:
$$A-p^2A-(1-p)^2A=2p(1-p)A.$$

Moreover, the mean value theorem for integration states that there exists $c\in (0,1)$ such that $\displaystyle{A=\int_0^1 f(x)\ dx=f(c)}$. Consequently, the area of the $f$-belos is
$$2p(1-p)f(c),$$
where $f(c)$ is the mean value of $f$ on $[0,1]$.

Since, by the above consideration, $\textrm{area}(\mathcal{P}(c))=p(1-p)f(c)$, the result follows.
\end{proof}

We will see how this property in fact generalizes Property 3 in \cite{SON}.

\begin{rem}
Consider a parbelos; i.e., an $f$-belos with $f(x)=x-x^2$. In this context Property 3 in \cite{SON} states that
$${\rm area}({\rm parbelos})=\frac{4}{3}{\rm area}(\mathcal{P}(1/2)).$$
Let us see how this follows from Property 4 above.

The mean value of $f$ in $(0,1)$ is $1/6$. Consequently, ${\rm area}(\mathcal{P}(c))=\dfrac{p(1-p)}{6}$. On the other hand, while $f(1/2)=1/4$, we have that ${\rm area}(\mathcal{P}(1/2))=\dfrac{p(1-p)}{4}$. Hence:
$${\rm area}({\rm parbelos})=2{\rm area}(\mathcal{P}(c))=\frac{4}{3}{\rm area}(\mathcal{P}(1/2))$$
as claimed.
\end{rem}

\section{The tangent parallelogram}

Throughout this section we will consider an $f$-belos such that $f$ is also differentiable in $x=0$ and $x=1$. Hence, we consider the tangents to $f$ in $x=0$ and in $x=1$ and the tangents in $x=p$ to $g$ and $h$ and a situation like Figure \ref{tan} makes sense.

\begin{figure}[h]
\label{tan}
\begin{center}
\includegraphics{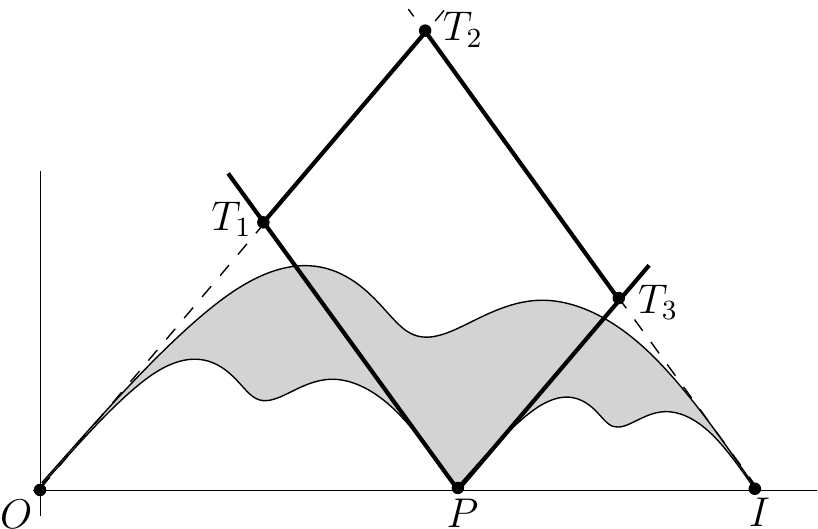}
\end{center}
\caption{The tangent parallelogram of an $f$-belos.}
\end{figure}

The fact that $T_1T_2T_3P$ is a parallelogram (that we will denote by $\mathcal{T}$) follows readily from the similarity of $f$, $g$ and $h$ (observe that $g'(p)=h'(0)$ and that $g'(1)=f'(1)$). In the following property we compute its area.

\begin{prop}
The area of $\mathcal{T}$ is given by:
$${\rm area}(\mathcal{T})=p(1-p)\left|\frac{f'(0)f'(1)}{f'(0)-f'(1)}\right|.$$
\end{prop}
\begin{proof}
Since the sides of the parallelogram are given by the equations
$$\overline{T_1T_2}:\ y=f'(0)x,$$
$$\overline{PT_3}:\ y=f'(0)(x-p),$$
$$\overline{T_1P}:\ y=f'(1)(x-p),$$
$$\overline{T_2T_3}:\ y=f'(1)(x-1),$$
it is easy to find the coordinates of $T_i$ for $i=1,2,3$ and obtain $\vec{T_1T_2}=\dfrac{(1-p)f'(1)}{f'(1)-f'(0)}(1,f'(0))$ and $\vec{T_1P}=\dfrac{-pf'(0)}{f'(1)-f'(0)}(1,f'(1))$. With these, we have that
$${\rm area}(\mathcal{T})=||\vec{T_1T_2}\times\vec{T_1P}||=\left|\left|\frac{p(1-p)f'(0)f'(1)}{(f'(1)-f'(0))^2}(0,0,f'(0)-f'(1))\right|\right|$$
and the result follows.
\end{proof}

In the parbelos , $\mathcal{T}$ is in fact a rectangle \cite[Property 4]{SON}. The general situation is as follows.

\begin{prop}
The tangent parallelogram $ \mathcal{T}$ is a rectangle if and only if $f'(0)f'(1)=-1$. In such case we have that
$${\rm area}(\mathcal{T})=p(1-p)\frac{f'(0)}{1+f'(0)^2}.$$
\end{prop}
\begin{proof}
The first part is straightforward recalling that two lines of slopes $m_1$ and $m_2$ are perpendicular if and only if $m_1m_2=-1$. For the second statement it is enough to put $f'(1)=-1/f'(0)$ in the previous property.
\end{proof}

Since we have that $2f'(0)\leq 1+f'(0)^2$, the following property follows (recall the notation from the previous section).

\begin{prop}
Given an $f$-belos, let $f(c)$ be the mean value of $f$ on $[0,1]$. Let also $\mathcal{P}(c)$ be the parallelogram associated to $c$ and $\mathcal{T}$ the tangent parallelogram. If $f'(0)f'(1)=-1$; i.e., if $\mathcal{T}$ is a rectangle, then:
$${\rm area}(\mathcal{P}(c))\geq 2f(c){\rm area}(\mathcal{T});$$
and equality holds if and only if $f'(0)=-f'(1)=1$.
\end{prop}
\begin{proof}
Properties 4 and 6 imply that
$${\rm area}(\mathcal{P}(c))=f(c)p(1-p)=f(c)\frac{{\rm area}(\mathcal{T})}{\frac{f'(0)}{1+f'(0)^2}}\geq 2f(c){\rm area}(\mathcal{T}).$$
And equality holds if and only if $2f'(0)=1+f'(0)^2$; i.e., if and only if $f'(0)=1$.
\end{proof}

\begin{rem}
In the parbelos $f(x)=x-x^2$, so we have $f'(0)=1=-f'(1)$ and equality holds in Property 7. Moreover (recall the remark after Property 4), $f(c)=1/6$. Hence:
$${\rm area}(\mathcal{T})=\frac{{\rm area}(\mathcal{P}(c))}{2f(c)}=\frac{3}{2}{\rm area}(\rm parbelos),$$
as was already proved in \cite[Property 4]{SON}.
\end{rem}

Now, if $f'(0)f'(1)=-1$, $\mathcal{T}$ is a rectangle, so it makes sense to consider its circumcircle $\Gamma$. In \cite[Property 6]{SON} it was proved that, in the parbelos case (i.e., when $f(x)=x-x^2$) the circumcircle of the tangent rectangle passes through the focus of the upper parabola (i.e., through the point $(1/2,0)$). This property can be generalized in the following way.

\begin{prop}
Given an $f$-belos such that its tangent parallelogram $\mathcal{T}$ is a rectangle; i.e., such that $f'(0)f'(1)=-1$, the circumcircle $\Gamma$ of $\mathcal{T}$ intersects the axis $OX$ at the point $(p,0)$ and $\displaystyle{\left(\frac{1}{1+f'(0)^2},0\right)}$ (see Figure \ref{circ}). Consequently:
\begin{itemize}
\item $\Gamma$ is tangent to $OX$ if and only if $\displaystyle{p=\frac{1}{1+f'(0)^2}}$.
\item $\Gamma$ intersects $OX$ at $x=1/2$ if and only if $f'(0)=1$.
\end{itemize}
\end{prop}
\begin{proof}
The center of this circle is the midpoint of $\overline{T_1T_3}$; i.e.:
$$C_{\Gamma}=\left(\frac{p+1+pf'(0)^2}{2(1+f'(0)^2)},\frac{f'(0)}{2(1+f'(0)^2)}\right).$$
This circle clearly intersects $OX$ at $P$ but, of course, it will intersect the axis $OX$ in another point (unless $\Gamma$ is tangent to $OX$). Hence we are looking for $p\neq\alpha\in(0,1)$ such that $(\alpha,0)\in \Gamma$. This condition leads to:
$$\left(\frac{p+1+pf'(0)^2}{2(1+f'(0))^2}-\alpha\right)^2+\left(\frac{f'(0)}{2(1+f'(0))^2}\right)^2=\left(\frac{p+1+pf'(0)^2}{2(1+f'(0))^2}-p\right)^2+\left(\frac{f'(0)}{2(1+f'(0))^2}\right)^2.$$
Consequently
$$(\alpha^2-p^2)-(\alpha-p)\frac{p+1+pf'(0)^2}{2(1+f'(0)^2)}=0$$
and, since $p\neq\alpha$,
$$\alpha=\frac{p+1+pf'(0)^2}{2(1+f'(0)^2)}-p=\frac{1}{1+f'(0)^2}$$
as claimed.
\end{proof}

\begin{figure}[h]
\label{circ}
\begin{center}
\includegraphics{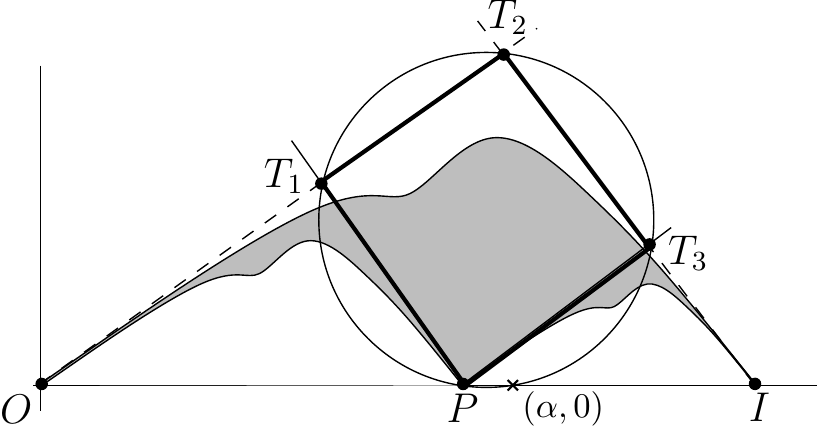}
\end{center}
\caption{The circumcircle of $\mathcal{T}$.}
\end{figure}

Now, we turn again to the general case, when $\mathcal{T}$ is ``only'' a parallelogram. The equation of the line passing through $T_1$ and $T_3$ is:
$$\overline{T_1T_3}:\ [(1-2p)f'(0)f'(1)]x+[pf'(0)-(1-p)f'(1)]y+p^2f'(0)f'(1)=0.$$
We are interested in studying when this line is tangent to the graph of $f$ at $(p,f(p))$. This interest is motivated by \cite[Property 5]{SON}, where it was proved that, in a parbelos, the diagonal of the tangent rectangle opposite the cust is tangent to the upper parabola at the point $(p,f(p))$. In fact we will see how this property characterizes parbeloses. 

\begin{prop}
The point $(p,f(p))$ lies on $\overline{T_1T_3}$ (see Figure \ref{tan}) if and only if
$$f(p)=\frac{p(p-1)f'(0)f'(1)}{pf'(0)-(1-p)f'(1)}.$$
In addition, the line $\overline{T_1T_3}$ is tangent to the graph of $f$ at the point $(p,f(p))$ if and only if 
$$f'(p)=\frac{(2p-1)f'(0)f'(1)}{pf'(0)-(1-p)f'(1)}.$$
\end{prop}
\begin{proof}
For the first part it is enough to substitute $(p,f(p))$ in the equation of $\overline{T_1T_3}$. For the second part, the slope of $\overline{T_1T_3}$ must coincide with $f'(p)$.
\end{proof}

\begin{figure}[h]
\label{tan}
\begin{center}
\includegraphics[width=\linewidth]{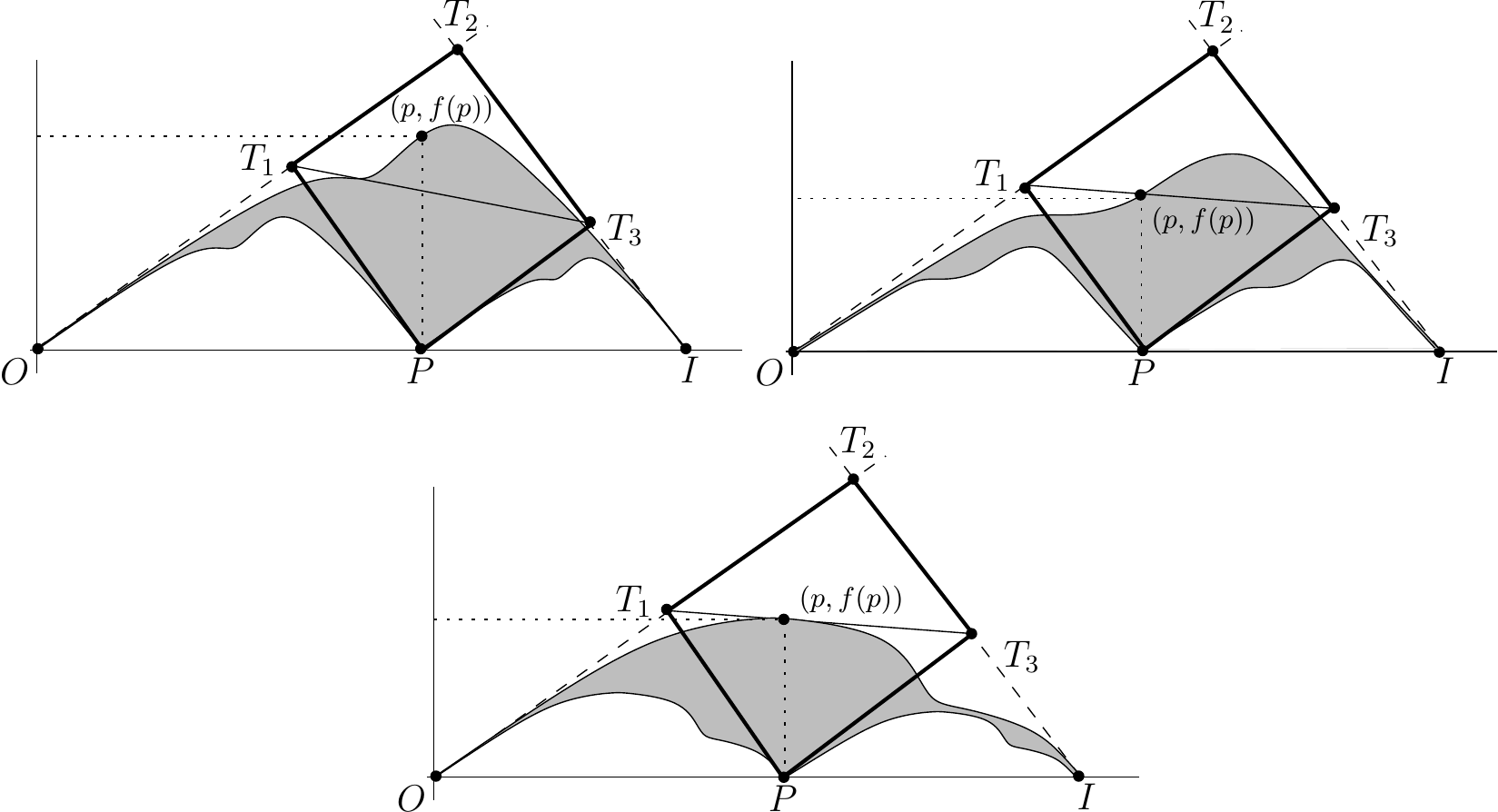}
\end{center}
\caption{Different relative positions of $f$ and $\overline{T_1T_3}$ at $(p,f(p))$.}
\end{figure}

We will close the paper with a property that surprisingly characterizes the parbelos among all possible $f$-beloses. Parbeloses are the only $f$-beloses such that the diagonal of $\mathcal{T}$ opposite to the cusp is tangent to $f$ at $(p,f(p))$ for every $p\in(0,1)$.

\begin{prop}
Given an $f$-belos, let $\mathcal{T}$ be its tangent parallelogram. Then the diagonal of $\mathcal{T}$ opposite to the cusp is tangent to $f$ at $(p,f(p))$ for every $p\in(0,1)$ if and only if $f$ is a parabola $f(x)=k(x-x^2)$. Moreover, in this case $\mathcal{T}$ is a rectangle if and only if $k=1$.
\end{prop}
\begin{proof}
The first part of the previous property leads to
$$f(x)=\frac{x(x-1)f'(0)f'(1)}{xf'(0)-(1-x)f'(1)}$$
for every $x\in[0,1]$. In this case we have that
$$f'(x)=\frac{f'(0)f'(1)[x^2f'(0)+(1-x)^2f'(1)]}{[xf'(0)-(1-x)f'(1)]^2}.$$
Hence, if $f$ is tangent to $\overline{T_1T_3}$ at $(p,f(p))$ for every $p\in(0,1)$, the second part of the previous property implies that:
$$f'(x)=\frac{f'(0)f'(1)[x^2f'(0)+(1-x)^2f'(1)]}{[xf'(0)-(1-x)f'(1)]^2}=\frac{(2x-1)f'(0)f'(1)}{xf'(0)-(1-x)f'(1)}$$
for every $x\in (0,1)$.
Some computations lead to
$$(f'(0)+f'(1))(x-x^2)=0$$
for every $x\in[0,1]$ and hence $f'(0)=-f'(1)$ so
$$f(x)=\frac{-x(x-1)f'(0)^2}{xf'(0)+(1-x)f'(0)}=f'(0)(x-x^2),$$
and the result follows.
\end{proof}

\section*{Acknowledgments}

I want to thank Jonathan Sondow for his many observations that have improved the paper.

\bibliography{./refarb}
 \bibliographystyle{plain}

\end{document}